\documentclass{article}

\usepackage{cite}
\usepackage{amsmath,amssymb,amsfonts}
\usepackage{algorithmic}
\usepackage{graphicx}
\usepackage{textcomp}
\usepackage{xcolor}
\usepackage{bbm}
\usepackage{amsthm}
\usepackage{subcaption}
\usepackage{mwe}
\usepackage{url}
\usepackage{multirow}
\usepackage{float}
\usepackage{dsfont}

\newcommand\numberthis{\addtocounter{equation}{1}\tag{\theequation}}

\newcommand\pf{{\operatorname{p-f}}}

\newtheorem{innercustomthm}{Theorem}
\newenvironment{thm1}[1]
	{\renewcommand\theinnercustomthm{#1}\innercustomthm}
	{\endinnercustomthm}
	
\newtheorem{definition}{Definition}
\newtheorem{condition}{Condition}
\newtheorem{theorem}{Theorem}
\newtheorem{lemma}{Lemma}

\begin{document}

\title{Convergence of Conditional Entropy for Long Range Dependent Markov Chains} 
\author{Andrew Feutrill and Matthew Roughan}
\maketitle



\begin{abstract}
In this paper we consider the convergence of the conditional entropy to the entropy rate for Markov chains. Convergence of certain statistics of long range dependent processes, such as the sample mean, is slow. It has been shown in Carpio and Daley~\cite{carpio2007long} that the convergence of the $n$-step transition probabilities to the stationary distribution is slow, without quantifying the convergence rate. We prove that the slow convergence also applies to convergence to an information-theoretic measure, the entropy rate, by showing that the convergence rate is equivalent to the convergence rate of the $n$-step transition probabilities to the stationary distribution, which is equivalent to the Markov chain mixing time problem. Then we quantify this convergence rate, and show that it is $O(n^{2H-2})$, where $n$ is the number of steps of the Markov chain and $H$ is the Hurst parameter. Finally, we show that due to this slow convergence, the mutual information between past and future is infinite if and only if the Markov chain is long range dependent. This is a discrete analogue of characterisations which have been shown for other long range dependent processes.
\end{abstract}






\section{Introduction}

Long range dependence (LRD) is a phenomenon that is associated with strong correlations between the present and past values of a stochastic process. It is often defined via its second-order properties, such as covariance or spectral density. Let $\{X_n\}_{n \in \mathbb{Z}^+}$ be a stationary stochastic process, then the covariance function is defined as $\gamma(k) = \mathds{E}\left[\left(X_{n+k} - \mu\right)\left(X_n - \mu\right)\right]$, for $k \in \mathbb{Z}^+$, where $\mu$ is the expected value of the process. We define a stationary stochastic process to be LRD if the auto-covariance function is not summable, \emph{i.e.},
\begin{align}
\sum_{k=1}^{\infty} \gamma(k) &= \infty. \label{eqn:def_lrd}
\end{align}

An important quantity which describes the degree of long range dependence, is the Hurst parameter, $H$, defined by
\begin{align*}
H &= \inf \left\{h : \limsup\limits_{n \rightarrow \infty} \frac{\sum_{r=1}^{n} \gamma(r)}{n^{2h-1}} < \infty \right\}.
\end{align*}
For LRD processes, $H \in (\frac{1}{2}, 1)$, with $H = \frac{1}{2}$ being short-range dependent and $H = 1$ being a totally, positively-correlated process.

As a result of LRD, some quantities of interest, such as the sample mean, are difficult to estimate from observed data, due to the fact that the estimators converge slowly to the true value~\cite{beran1992statistical}. For example, the variance of the sample mean decreases at the rate $O(n^{2H-2})$ for LRD processes, and at the rate of $O(n^{-1})$ for non-LRD processes~\cite{beran1992statistical}, where $n$ is the number of samples. Slow convergence has been shown to characterise the behaviour of LRD processes for many quantities. 

LRD processes defined on discrete state spaces in discrete-time have been much less studied than on continuous state spaces but the concept was introduced into general point processes~\cite{daley1997long}, renewal~\cite{daley1999hurst} and Markov renewal processes~\cite{vesilo2004long}. We consider the question of whether the slow convergence applies to an information-theoretic measure, the entropy rate. Intuitively, this measures the average information gain from sampling a new random variable in a stochastic process and therefore we expect similar behaviour in its convergence rate and indeed we show this to be true here. 

Our contributions are as follows:
\begin{enumerate}
	\item We analyse the convergence of the conditional entropy to the entropy rate, and show that the convergence rate is the same as the Markov chain mixing time problem, that is, the convergence of the $n$-step transition probabilities to the stationary distribution.
	\item We quantify what slow convergence means for LRD Markov chains and prove that the convergence rate to the entropy rate is $O(n^{2H-2})$.
	\item  We show that this slow convergence rate provides a characterisation of LRD on Markov chains as having infinite mutual information between past and future. This can be interpreted as an infinite amount of information being shared between the past and future of the process.
\end{enumerate}

 This is similar to the work of Feutrill and Roughan~\cite{feutrill2021differential}, that showed that the conditional entropy converges to the entropy rate slower for LRD discrete-time Gaussian processes than their short-range-dependent counterparts. However, here we consider Markov chains on discrete state spaces which are more relevant to many real contexts, such as the analysis of natural language.

\section{Preliminaries and Main Results}

We briefly introduce some of the concepts required in the analysis of Markov chains and LRD processes. We begin with defining the Markov chain transition probabilities. Let $\{X_n\}_{n \in \mathbb{Z}^+}$ be an ergodic Markov chain with state space $S$. We denote the one-step transition probability between states $i$ and $j \in S$ by $p_{ij} = \mathds{P}(X_n = j | X_{n-1} = i)$, and the $n$-step transition probabilities, by  $p_{ij}^n = \mathds{P}(X_{k+n} = j | X_k = i)$, for $n \ge 0$.

In addition to the standard definition of LRD, given in (\ref{eqn:def_lrd}), a key insight for point and renewal processes is that LRD can be defined with respect to the second-order behaviour of the counting function of the number of events in an interval, $N(0,t]$. Then, a definition of LRD for point and renewal processes is given by the variance of this function. A point or renewal process is said to have LRD if the growth of the variance of the function is faster than linear, that is
\begin{align*}
\limsup \limits_{t \rightarrow \infty} \frac{\mbox{Var}(N(0,t])}{t} &= \infty.
\end{align*}
This was extended to the case of irreducible Markov chains in discrete time on countable state spaces by Carpio and Daley~\cite{carpio2007long}, by using the variance of the counting function of the number of visits of the Markov chain, $\{X_n\}_{n \in \mathbb{Z}^+}$, to state $i$ up until time $n$, $N_i(0,n]$, that is the variance of
\begin{align*}
N_i(0,n] = \frac{1}{n} \sum_{k = 1}^{n} \mathbbm{1}_{\{X_k = i\}}.
\end{align*}
This is a natural extension, since the rate of increase of the variance of the counting function is a property of the communicating class. This leads to the following definition of LRD on Markov chains.
\begin{definition}
	An irreducible and aperiodic Markov chain, $\{X_n\}_{n \in \mathbb{Z}^+}$, is said to be long range dependent if
	\begin{align*}
	\limsup \limits_{n \rightarrow \infty} \frac{\mbox{Var}(N_i(0,n])}{n} &= \infty.
	\end{align*}
	Otherwise, we say the Markov chain is short range dependent~\cite{carpio2007long}.
\end{definition}
Note that from Carpio and Daley~\cite{carpio2007long} that this applies to all states in the communicating class, and therefore is independent of the particular state $i$ used in the definition.

A random variable which characterises the long-term behaviour of Markov chains is the return time random variable to a particular state $i$, $T_{ii}$,which is defined as $T_{ii} = \inf\{n \ge 1: X_n = i, X_0 = i\}$. LRD Markov chains were shown, in Lemma 1 of Carpio and Daley~\cite{carpio2007long}, to have a return time random variable with an infinite second moment, providing a second useful characterisation of LRD that is used in the analysis in this paper. Oguz and Anantharan~\cite{ouguz2012hurst} extended this result to functions of Markov chains which inherit LRD, while preserving the Hurst parameter, given some additional regularity conditions on the function.

We define a condition that we will use throughout the paper, concerning the tail decay of the return time random variable.
\begin{condition}\label{cond:power_law_ccdf}
	The complementary cumulative distribution function of the return time random variable $T_{ii}$ has a power-law tail. That is, $\mathds{P}(T_{ii} > n) \sim n^{-\alpha}$, where $1 < \alpha < 2$.
\end{condition}

A related concept to infinite moments is that of heavy-tailed distributions, which are distributions whose tail decay is slower than exponential. 

\begin{definition}\label{def:heavy_tail}
	A probability distribution of a random variable, $X$, with cumulative distribution function, $F$, is called heavy-tailed if and only if $\forall t > 0$
	\begin{align*}
	\int_{-\infty}^{\infty} e^{t x} dF(x) = \infty.
	\end{align*}
\end{definition}
We will discuss some of the aspects of Markov chain stationary distribution convergence with reference to return time distributions with this property. Note that in this paper we will be considering \emph{discrete} random variables with heavy tails, and the integral will become an infinite sum.

We now define the quantity of interest in this paper, the \emph{entropy rate}. The entropy rate is the asymptotic limit of the average information from each additional observation of the Markov chain, and is used as a measure of uncertainty or complexity of a stochastic process.

\begin{definition}
	For a stochastic process, $\mathcal{X} = \{X_i\}_{i \in \mathbb{Z}^+}$, the entropy rate, is defined as,
	\begin{align*}
	\mathds{H}(\mathcal{X}) &= \lim\limits_{n \rightarrow \infty} \frac{1}{n} \mathds{H}(X_0, ... , X_{n-1}),
	\end{align*}
	when the limit exists, where $\mathds{H}(X_0, ... , X_{n-1})$ is the joint entropy of $X_0, \ldots, X_{n-1}$.
\end{definition}

Now we consider some additional information theoretic quantities, the excess entropy, which considers the ``excess" information that accrues in the convergence to the entropy rate from the conditional entropy and the mutual information between past and future. It measures the amount of information that is shared between the infinite past and infinite future of processes. For processes on countable sets, the excess entropy was shown to be equivalent to the mutual information between past and future~\cite{crutchfield_feldman_2003}. We define both here and show that in the case of LRD Markov chains that these two measures are infinite, providing another characterisation of LRD. Note that we denote the excess entropy, $E$, in line with previous work and use the expectation operator $\mathds{E}[\cdot]$.

\begin{definition}
	The \emph{excess} entropy, $E$, of a stochastic process, $\{X_i\}_{i \in \mathbb{Z}^+}$, is defined as,
	\begin{align*}
	E &= \sum_{n=1}^{\infty} \Big(\mathds{H}[X_n | X_{n-1}, \ldots, X_0] - \mathds{H}(\mathcal{X})\Big),
	\end{align*}
	where $\mathds{H}[X_n|\cdot]$ is the conditional entropy.
\end{definition}

We also define the mutual information between past and future. 

\begin{definition}
	The mutual information between past and future, $I_\pf$, is defined as 
	\begin{align*}
	I_\pf &= I(\{X_s, s < 0\}, \{X_s, s \ge 0\}),
	\end{align*}
	where $I(\cdot, \cdot)$ is the mutual information.
\end{definition}
This measures the amount of shared information between the infinite future and the infinite past of a stochastic process.

In addition to the concepts above we need a notion of distance between probability mass functions. Specifically, the distance between the $n$-step transition probabilities and the stationary distribution. Here we use the total variation norm.
\begin{definition}
	The total variation distance between two probability distributions, $\mu$ and $\nu$ on a support, $S$, with an associated sigma algebra, $\mathcal{F}$ is defined as
	\begin{align*}
	d(\mu, \nu) = ||\mu - \nu||_{TV} &= \sup_{A \in \mathcal{F}} |\mu(A) - \nu(A)|.
	\end{align*}
\end{definition}
This definition gives the total variation as the maximum difference between the two probability distributions across all possible events. This is linked to the $L^1$-norm by
\begin{align*}
2|\mu - \nu| = ||\mu - \nu||_{TV},
\end{align*}
for all distributions~\cite[pg. 314]{meyn_tweedie_glynn_2009}.
An extension of this distance is called the $f$-norm, which is used in the statements of more general convergence theorems where convergence is of the quantity $f\left(X_n\right)$ to its mean value, given an arbitrary function $f : S \rightarrow [1, \infty)$ of the states.
\begin{definition}
	The $f$-norm of two probability distributions, $\mu$ and $\nu$ on a support $S$ is defined as
	\begin{align*}
	||\mu - \nu||_{f} &= \sup_{g: |g| \le f} |\mu(g) - \nu(g)|,
	\end{align*}
	where $\mu(g) = \sum_{i \in S} \mu(i)g(i)$, for an arbitrary function $g$. 
\end{definition}
Note that the definition here applies to any function $g$ that is dominated by $f$.
The total variation and f-norms are equivalent using the function $f=1$~\cite{tuominen_tweedie_1994}.


We now summarise the main results of the paper, three theorems describing the convergence behaviour of LRD Markov chains and the impact on information theoretic quantities. 

The first theorem concerns the convergence rate of the $n$-step transition probabilities to the stationary distribution of LRD Markov chains, on a state space $S$, that have a power-law decaying complementary cumulative distribution function for the return time, $T_{ii}$.

\begin{theorem}\label{thm:power_law_markov_chain_convergence}
	Let $\mathcal{X} = \{X_n\}_{n \in \mathbb{Z}^+}$ be an LRD Markov chain with Condition~\ref{cond:power_law_ccdf}. Then, the rate of convergence of the $n$-step transition probabilities to the stationary distribution $\pi$ is $O(n^{2H-2})$. That is, for any $0 < \beta < H$ and $\forall i \in S$
	\begin{align*}
	n^{2-2\beta} || p^n_{i \cdot} - \pi\left(\cdot\right) ||_{TV} \rightarrow 0.
	\end{align*}
\end{theorem}

The next result gives the link between the convergence rate of the $n$-step transition probabilities to the stationary distribution.

\begin{theorem}\label{thm:stationary_entropy_rate_convergence}
	The convergence of the conditional entropy of an ergodic, positive recurrent Markov Chain to its entropy rate is at the same rate as the convergence to the stationary distribution.
\end{theorem}

We now present our final result, which gives a characterisation of LRD on Markov chains that is similar to that given by Feutrill and Roughan~\cite{feutrill2021differential} for Gaussian processes.

\begin{theorem}\label{thm:mutual_infinite}
	The mutual information between past and future of a Markov chain is infinite if and only if the Markov chain is LRD.
\end{theorem}

\section{Convergence to the Stationary Distribution for Markov Chains}
In this section we  introduce some of the main concepts and relevant results regarding the convergence of the limit of $n$-step transition probabilities to the stationary distribution for Markov chains. This subject has been well studied, in particular conditions where the Markov chain converges at a geometric rate, \emph{i.e.,} decays as $\rho^n$ for some $\rho$ such that $0<\rho<1$, are well known. It was also noted in Carpio and Daley~\cite{carpio2007long} that the convergence of the $n$-step transitions probabilities to the stationary distribution is ``slow" for LRD processes. 

A classic theorem, Theorem~\ref{thm:finite_convergence} below, classifies the convergence rate of all finite state Markov chains. Characterising the convergence of finite-state Markov chains is simpler than the countable state case, as every recurrent Markov chain is positive recurrent and all moments of the return time are finite~\cite[Theorem 7.3.1]{hunter2014mathematical}. 
\begin{theorem}[Theorem 4.9~\cite{levin2017markov}]\label{thm:finite_convergence}
	For an ergodic Markov chain on a finite state space $S$ with stationary distribution $\pi$ and probability transition matrix $P$, there exists $\alpha \in (0,1)$ and $C > 0$ such that $\forall i, j \in S$
	\begin{align*}
	\max_{i \in S} || p^n_{ij} - \pi_j ||_{TV} \le C \alpha^n.
	\end{align*}
\end{theorem}

Geometric convergence for Markov chains on countably infinite state spaces requires more conditions, since moments of the return time can be infinite. An important concept, introduced by Kendall~\cite{kendall1959unitary}, is the following.

\begin{definition}[Geometric Ergodicity]
	A Markov chain, $\{X_n\}_{n \in \mathbb{Z}^+}$, is geometrically ergodic if there exist numbers $c_i, \pi_i$ and $0 \le \rho_i < 1$ for every state $i \in S$ such that $\forall n \in \mathbb{Z}^+$
	\begin{align*}
		||p_{ii}^n - \pi_i||_{TV} &\le c_i\rho_i^n.
	\end{align*}
\end{definition}


Geometric ergodicity implies a fast convergence rate, as it can be bounded by an exponentially decaying function. However, LRD processes have heavy tailed probability distributions of return times, with infinite moments, which we show in this section precludes geometric convergence. 


An extension of Theorem~\ref{thm:finite_convergence} has been developed for Markov chains on general, not necessarily countable, state spaces. It requires additional definitions to state its conditions.  First we will define the sampled chain of a Markov chain. 

\begin{definition}
	Let $a = \{a_n\}_{n \in \mathbb{Z}^+}$ be a distribution, then we define the sampled chain of a Markov chain $\{X_n\}_{n \in \mathbb{Z}^+}$ for a state $i$ and a subset of the state space, $A$, to be
	\begin{align*}
	K_a(i, A) &= \sum_{n=0}^{\infty} p^n_{iA} a_n,
	\end{align*}
	where $p^n_{iA}$ is the $n$-step transition probability of moving from state $i$ to a subset $A$.
\end{definition}

\noindent Now we define the concept of a petite set.

\begin{definition}
	A set $C$ is called petite if the sampled chain satisfies the following bound
	\begin{align*}
	K_a(i, A) \ge \nu(A),
	\end{align*}
	for all $i \in C$ and for all subsets $A$, and for a non-trivial measure $\nu$, that is $\nu(A) \neq 0$.
\end{definition}

\noindent When $S$ is countable, every state $i \in S$ forms a singleton petite set and we use these results from general state spaces to refer to petite sets of a single countable state.




The following theorem summarises some important implications and characterisations of geometric ergodicity. The notation $p^\infty(C)$ is the limiting probability of being in a subset $C$.

\begin{theorem}[Geometric Ergodic Theorem{~\cite[Theorem 15.0.1]{meyn_tweedie_glynn_2009}}]\label{thm:geometric_convergence_general}
	For an ergodic Markov chain, $\{X_n\}_{n \in \mathbb{Z}^+}$, on a countable state space, the following conditions are equivalent:
	\begin{enumerate}
		\item The chain $\{X_n\}_{n \in \mathbb{Z}^+}$ is positive recurrent with a stationary distribution and there exists a petite set $C, 0 < \rho_C < 1$, $0 < M_C < \infty$ and $p^\infty(C) > 0$, such that for all $i \in C$
		\begin{align*}
			|p^n_{i, C} - p^\infty(C)| \le M_C \rho_C^n.
		\end{align*}
		\item There exists some petite set $C$ and $\gamma>1$ for all $i \in C$ such that
		\begin{align*}
			\sup_{i \in C} \mathds{E}_i[\gamma^{T_{ii}}] < \infty,
		\end{align*}
		where $\mathds{E}_i[\cdot] = \mathds{E}[\cdot| X_0 = i]$. \label{thm:part_2_geometric_ergodicity_general}
	\end{enumerate}
\end{theorem}
Note that we have removed some equivalent conditions that are irrelevant to our discussion.



 
Part~\ref{thm:part_2_geometric_ergodicity_general} of Theorem~\ref{thm:geometric_convergence_general} is a condition on the radius of convergence of the probability generating function of the return time random variable. When considering sets consisting of a single point, it reduces to a condition on the return time distribution. We define the probability generating function of the return time distribution as
\begin{align*}
F_{ii}(z) = \sum_{n=1}^{\infty} \mathds{P}(T_{ii}  = n) z^n = \mathds{E}[z^{T_{ii}}].
\end{align*}
for $z \in \mathbb{C}$. By Theorem~\ref{thm:geometric_convergence_general}, if the radius of convergence of $F_{ii}(z)$ is greater than 1, then the chain is geometrically ergodic. For a positive recurrent Markov chain, the radius of convergence is at least 1 since the return time is finite almost surely, and hence $\sum_{n=1}^{\infty} \mathds{P}(T_{ii}  = n)  = 1$.


\begin{lemma}\label{infinite_moment}
	The return time distribution of a Markov chain is heavy-tailed if and only if the convergence to the stationary distribution is slower than geometric.
\end{lemma}

\begin{proof}
	The definition of a heavy-tailed distribution, is equivalent the moment generating function, $E[e^{tT_{ii}}]$ being infinite for all $t > 0$ because a heavy tail from Definition~\ref{def:heavy_tail} is equivalent to having infinite moments~\cite[pg. 11]{foss2011introduction}. This implies that for any heavy-tailed return time distribution, that $F_{ii}(e^t) = \infty \ \forall \ t > 0 \iff \mathds{E}[e^{t T_{ii}}] = \infty, \ \forall \ t > 0$. For the probability generating function, $F_{ii}(z)$, $t > 0 \iff z > e^0 = 1$. Which implies that the radius of convergence is exactly 1.
\end{proof}

Hence, LRD Markov chains must converge more slowly than geometrically. We now use the knowledge of the moments of the return time to provide a convergence rate for LRD processes and  introduce similar results to geometric ergodicity but for general rate functions.


The direct analogue of the classification of geometric ergodicity for general rate functions is given by the following theorem. 

\begin{theorem}[Theorem 2.1~\cite{tuominen_tweedie_1994}]\label{thm:generic_rate_convergence}
	For an ergodic Markov chain, a function $f:S \rightarrow [1, \infty)$ and a rate function, $r(n):\mathbb{Z}^+ \rightarrow \mathbb{R^+}$, the following statements are equivalent:
	\begin{enumerate}
		\item There exists a petite set, $C$, such that
		\begin{align*}
		\sup_{i \in C} \mathds{E}_i \left[\sum_{k=0}^{T_C - 1}r(k)f(X_k)\right] < \infty,
		\end{align*}
		where $T_C$ is the return time to the set $C$ and $\mathds{E}_i[\cdot] = \mathds{E}[\cdot | X_0 = i]$.
		\item The sequence, $r(n)|| p^n_{i,\cdot} - \pi\left(\cdot\right)||_f \rightarrow 0$ as $n \rightarrow \infty$ for all $C$ such that
		\begin{align*}
		\sup_{i \in C} \mathds{E}_i \left[\sum_{k=0}^{T_B-1} r(k)f(X_k)\right] < \infty, 
		\end{align*}
		for all subsets $B \in \mathcal{F}$ where $T_B$ is the return time to subset $B$.
	\end{enumerate}
\end{theorem}





We have shown in Lemma~\ref{infinite_moment}, that any infinite moment of the return time distribution implies that the Markov chain cannot converge geometrically. Next, we show that we can link the maximum convergence rate of the Markov chain to the supremum of finite moments of the return time, under Condition~\ref{cond:power_law_ccdf}. 



To use the first part of Theorem~\ref{thm:generic_rate_convergence}, we require a lemma to that when $f$ is the constant function equal to 1, \emph{i.e.,} $f=1$ that $\mathds{E}_i\left[\sum_{k=0}^{T_{ii} - 1} r(k)\right] < \infty$ is equivalent to an easier to analyse expression to for power law decay. The behaviour of the return time random variable is a property of the communication class, and therefore for an irreducible Markov chain the power-law behaviour of the return time to a state $i$ applies to all states.

\begin{lemma}\label{alternate_condition}
	For an ergodic Markov chain where $\mathds{P}(T_{ii} > n) \sim n^{-\alpha}$,
	\begin{align*}
	\mathds{E}_i\left[\sum_{k=0}^{T_{ii} - 1} r(k)\right] < \infty,
	\end{align*}
	if and only if 
	\begin{align*}
	\sum_{k=1}^{\infty} r(k) k^{-\alpha} < \infty.
	\end{align*}
\end{lemma}

\begin{proof}
	Note that
	\begin{align*}
	\mathds{E}_i\left[\sum_{k=0}^{T_{ii} - 1} r(k)\right] &= \sum_{n=1}^{\infty} \left(\sum_{k=0}^{n - 1} r(k)\right) \mathds{P}(T_{ii} = n)\\
	&= \sum_{n=1}^{\infty} r(n) \mathds{P}(T_{ii} > n),
	\end{align*}
	where the second equality follows by Tonelli's theorem for a positive random variable.
\end{proof}



With this, we can state the next result, Lemma~\ref{lem:power_law_markov_chain_convergence}, that the rate of convergence, via the exponent of a power-law, is dependent on the existence of a corresponding moment of the return time random variable.

\begin{lemma}\label{lem:power_law_markov_chain_convergence}
	The rate of convergence of the n-step transition probabilities to the stationary distribution of a Markov chain with Condition~\ref{cond:power_law_ccdf}, is $O(n^{1-\alpha})$. That is, for any $0 < \beta < \alpha - 1$
	\begin{align*}
	n^\beta || p^n_{i, \cdot} - \pi\left(\cdot\right) ||_{TV} \rightarrow 0.
	\end{align*}
\end{lemma}

\begin{proof}
	From Theorem~\ref{thm:generic_rate_convergence}, we can show that the return time random variable in part 1 converges using the function $f = 1$. By part 3, we have that $r(n)|| p^n_{i,\cdot} - \pi||_{TV} \rightarrow 0$, since the total variation norm is the supremum of functions $g$ dominated by the constant 1,
	\begin{align*}
		||\mu - \pi||_{TV} &= \sup_{g: |g| \le 1} |\mu(g) - \pi(g)|.
	\end{align*}
	By Lemma~\ref{alternate_condition}, this is equivalent to showing conditions for
	\begin{align*}
		 \sum_{n=1}^{\infty} r(n) n^{-\alpha} < \infty,
	\end{align*}
	and we can apply Theorem~\ref{thm:generic_rate_convergence}, to show convergence.
	We can see that we require a function $r(n) = n^\beta$ such that, $\beta - \alpha < -1$, since any sum $\sum_{n=1}^{\infty} n^\gamma$, diverges for $\gamma \ge -1$.
	This implies that we require $\beta < \alpha - 1$ and also by Theorem~\ref{thm:generic_rate_convergence}, we require $\beta > 0$. So any rate between these two will converge.
\end{proof}

Condition~\ref{cond:power_law_ccdf} requires a return time with power law, \emph{e.g.} $\mathds{P}(T_{ii} > n) \sim n^{-\alpha}$, but there may be many values where this is true. We use
\begin{align*}
\alpha = \sup \left\{\delta : \mathds{E}[T_{ii}^\delta] < \infty \right\},
\end{align*}
which we call the moment index. From the previous discussion and Lemma~\ref{infinite_moment}, we can conclude that if all moments of the return time random variable exist, then the convergence to the stationary distribution is geometric. From Lemma~\ref{lem:power_law_markov_chain_convergence}, if there exist any infinite moments of the return time random variable and a power-law tail in the return time random variable, then the convergence is a power-law with exponent of the moment index minus 1, \emph{i.e.}, $\alpha - 1$.


This gives an interesting link between the convergence rate of an LRD Markov chain and the Hurst parameter. This is summarised in Theorem~\ref{thm:power_law_markov_chain_convergence}.

\begin{thm1}{1}
	Let $\mathcal{X} = \{X_n\}_{n \in \mathbb{Z}^+}$ be an LRD Markov chain with Condition~\ref{cond:power_law_ccdf}. Then, the rate of convergence of the $n$-step transition probabilities to the stationary distribution $\pi$ is $O(n^{2H-2})$. That is, for any $0 < \beta < H$ and $\forall i \in S$
	\begin{align*}
	n^{2-2\beta} || p^n_{i \cdot} - \pi\left(\cdot\right) ||_{TV} \rightarrow 0.
	\end{align*}
\end{thm1}

\begin{proof}[Proof of Theorem~\ref{thm:power_law_markov_chain_convergence}]
	From Carpio and Daley~\cite{carpio2007long} and Theorem 1 of Daley~\cite{daley1999hurst} we have that the Hurst parameter is linked to the moment index by the following relationship,
	\begin{align}
		H = \frac{1}{2} \left(3 - \alpha\right).\label{eq:hurst_exponent_relationship}
	\end{align}
	Since, the exponent $\alpha$ in the complementary cumulative distribution function represents the moment index in the distribution~\cite[pg. 32]{foss2011introduction}, and by rearranging (\ref{eq:hurst_exponent_relationship}) that $\alpha -1 = 2 - 2H$. Then the result follows by applying Lemma~\ref{lem:power_law_markov_chain_convergence}.
\end{proof}

This result echos previous work in the area of LRD, \emph{e.g.,}~\cite{beran1992statistical}, which shows that the convergence rate to important quantities is slower for LRD processes and is related to the Hurst parameter. Interestingly, this result illustrates that the convergence rate slows as the Hurst parameter tends to one, \emph{i.e.,} as the degree of LRD increases. As $H \rightarrow 1$ the exponent tends to $0$, and the moment index is close to 1 and at that stage the expectation of the return time is finite. If the moment index is $\le 1$, then the expectation becomes infinite, and the Markov chain is null-recurrent, \emph{i.e.}, $\mathds{E}[T_{ii}] = \infty$ and $\alpha = 1$. Intuitively, this indicates that the rate convergence slows with respect to the degree of LRD, and once the Hurst parameter becomes greater than 1, using the definition from the moment index, the rate of convergence slows to zero and the stationary distribution no longer exists.

\section{Convergence of conditional entropy to entropy rate}

In this section we discuss the convergence of the conditional entropy of a Markov chain to its entropy rate. We show that there is an equivalence between the rate of convergence of the $n$-step transition probabilities to the stationary distribution and the convergence of the conditional entropy to the entropy rate. From Carpio and Daley~\cite{carpio2007long} the convergence rate is a property of the entire communicating class, so any convergence rate in a particular state applies to all states in an ergodic chain. We demonstrate that the behaviour of LRD Markov chains is consistent with properties of other processes and conclude for ergodic Markov chains that LRD is characterised by slow convergence and an infinite amount of shared information between the past and future of the process. 

Next we state that the entropy rate of an ergodic Markov chain is the same as that of a stationary Markov chain. This aligns with the intuition of ergodic processes, that they ``forget" their initial state, and can be analysed using their asymptotic probabilities. We extend the result to a Markov chain starting from an arbitrary initial state, and omit the proof as it follows the same argument as Theorem 4.2.4 in Cover and Thomas~\cite{cover_thomas_2006}. We make one additional assumption, required for the calculation of the entropy rate over a countable set, that the conditional entropy given knowledge of which state the process is in, is finite; formally, that is 
\begin{align*}
	\mathds{H}\left[X_n|X_{n-1} = i\right] = - \sum_{j \in \Omega} p_{ij} \log p_{ij} < \infty.
\end{align*}
This is not an onerous assumption, since for interesting analysis we require that the entropy of a random variable is finite.

\begin{theorem}
	The entropy rate of an ergodic Markov Chain $\mathcal{X}$ is
	\begin{align*}
	\mathds{H}[\mathcal{X}] &= -\sum_{i} \sum_{j} \pi_i p_{ij} \log p_{ij}.
	\end{align*}
\end{theorem}

\begin{proof}
	The proof is omitted as it follows the same argument as Theorem 4.2.4 in Cover and Thomas~\cite{cover_thomas_2006}.
\end{proof}

Therefore we have an explicit form for the entropy rate of a Markov chain that is ergodic, rather than just stationary and therefore is the same as under any distribution of the initial state of the chain. 

Next, by analysing the limit of the conditional entropy conditioned on the previous observations of the Markov chain, we show that the convergence to the entropy rate is equivalent to the convergence to the stationary distribution. This provides another perspective on long range dependence, that the convergence to the entropy rate, the average new information from a random variable of a stochastic process, is slower. This equivalence is given by Theorem~\ref{thm:stationary_entropy_rate_convergence}, and this is proved below.


\begin{thm1}{2}
	The convergence of the conditional entropy of an ergodic, positive recurrent Markov Chain to its entropy rate is at the same rate as the convergence to the stationary distribution.
\end{thm1}

\begin{proof}[Proof of Theorem~\ref{thm:stationary_entropy_rate_convergence}]
	Define the initial distribution as $\mathbf{\alpha} = \{\alpha_i\}_{i \in S}$, \emph{i.e.}, $\mathds{P}(X_1 = i) = \alpha_i$. The conditional entropy of $X_2$ given $X_1$ is
	\begin{align*}
		\mathds{H}[X_2 | X_1] &= \sum_{i} \sum_{j} \mathds{P}(X_1 = i, X_2 = j) \log \mathds{P}(X_2 = j | X_1 = i),\\
		&= \sum_{i} \sum_{j} \alpha_i p_{ij} \log p_{ij}.
	\end{align*}
	Considering the conditional entropy of $X_n$ given the full history, $X_1, \ldots, X_{n-1}$, 
	\begin{align*}
		\mathds{H}[X_n &| X_{n-1}, \ldots, X_0] \\
		&= \sum_{i_n} \ldots \sum_{i_0} \mathds{P}(X_0 = i_0, \ldots, X_n = i_n) \log \mathds{P}(X_n = i_n | X_{n-1} = i_{n-1}, \ldots, X_0 = i_0),\\
		&= \sum_{i_n} \ldots \sum_{i_0} \alpha_{i_0} p_{i_0i_n}^{n+1} \log p_{i_{n-1}i_n},\\
		&= \sum_{i_n} \sum_{i_{n-1}} p_{i_{n-1} i_n} \log p_{i_{n-1} i_n} \left(\sum_{i_{n-1}} \ldots \sum_{i_0} \alpha_i p_{i_0 i_{n-1}}^n\right),\\
		&= \sum_{i_n} \sum_{i_{n-1}} p_{i_{n-1}i_n} \log p_{i_{n-1} i_n} \left(\sum_{i_0} \alpha_i p_{i_0 i_{n-1}}^n
		\right), \numberthis \label{eqn:conditional_entropy_from_pre_limit}
	\end{align*}
	where we can factor out the path from the initial state to the second to last state due to the Markov property in the last equality. Note that the term $p_{i_{n-1} i_n} \log p_{i_{n-1} i_n}$ quantifies the information contained in the transitions. As $n \rightarrow \infty$, $\sum_{i} \alpha_i p_{i,j}^n \rightarrow \pi_j$, since a positive recurrent chain has a stationary distribution and by the ergodicity of the chain, it converges to the stationary distribution from any state. 
	Taking the limit of Equation~(\ref{eqn:conditional_entropy_from_pre_limit}) shows that the convergence of the conditional entropy to the entropy rate depends on the rate of convergence of $\sum_{i} \alpha_i p_{ij}^n \rightarrow \pi_j$.
\end{proof}

Theorem~\ref{thm:stationary_entropy_rate_convergence} tells us that the convergence rate of the conditional entropy to the entropy rate is the same as convergence of the $n$-step transition probabilities to the stationary distribution. This shows that the convergence rate of other quantities for Markov chains are intimately connected to the convergence rate of the stationary distribution. 

A definition of LRD, suggested in Li~\cite{Li_2004}, is that the mutual information between past and future is infinite. This has been shown to be true for Gaussian processes~\cite{feutrill2021differential}, and we show it holds for Markov chains as an extension of the following lemma. We show that in the case of LRD the excess entropy is infinite, by the limits of the quantities, $Q_{ij}^n = \sum_{r=1}^{n} \left(p_{ij}^r - \pi_j\right)$, for $i,j \in \Omega$. Carpio and Daley~\cite{carpio2007long} used $Q_{ij}^n$ to show that the state space must be infinite in the case of LRD. We use it to illustrate the slow convergence, and to reinforce an entropic perspective on LRD. 

\begin{lemma}\label{lem:excess_infinite}
	A countable state Markov chain is LRD if and only if the excess entropy is infinite.
\end{lemma}


\begin{proof}
	We begin by considering an arbitrary term in the excess entropy sum and simplifying, which gives
	\begin{align*}
		\mathds{H}[X_n | X_{n-1}, \ldots, X_0] - \mathds{H}[\mathcal{X}] 
		&= - \sum_{i} \sum_{j} p_{ij} \log p_{ij} \left(\sum_{k} \alpha_k p_{ki}^n -  \pi_i\right).
	\end{align*}
	Then we consider partial sums up to order $n$ to get the excess entropy, 
	\begin{align*}
		E(n) &= \sum_{r = 1}^{n} \mathds{H}[X_r | X_{r-1}, \ldots, X_1] - \mathds{H}[\mathcal{X}],\\
		&= \sum_{r=1}^{n} \left(- \sum_{i} \sum_{j} p_{ij} \log p_{ij} \left(\sum_{k} \alpha_k p_{k, i}^r -  \pi_i\right)\right)\\
		&= - \sum_{i} \sum_{j} p_{ij} \log p_{ij} \sum_{r = 1}^{n} \left(\sum_{k} \alpha_k p_{k, i}^r -  \pi_i\right),
	\end{align*}
	where we can swap the order of terms due to the finite sum.
	
	Note that since we have a weighted average of the $n$-step transition probabilities, there exist some initial states that will be either larger or smaller than the average. So we can bound any initial distribution using $Q_{ij}^n$ for some arbitrary states, $i,j \in S$. Therefore, $\exists s,t \in S, \forall n$,
	\begin{align*}
		\sum_{r=1}^n p_{sj}^ n - \pi_j &\le \sum_{r=1}^n \sum_{k} \alpha_k p_{kj}^n - \pi_j \le \sum_{r=1}^n p_{tj}^ n - \pi_j,\\
		\intertext{and hence } 
		Q_{ij}^n &\sim \sum_{r=1}^n \sum_{k} \alpha_k p_{kj}^n - \pi_j.
	\end{align*}
	By Lemma 3 of Carpio and Daley~\cite{carpio2007long}, we have that the rate of growth is asymptotically the same for all combinations of states. Therefore calling this rate function, $f(n)$ gives
	\begin{align*}
		\frac{Q_{ij}^n}{\pi_j} &\sim f(n),\\
		\intertext{and hence} 
		\sum_{r=1}^{n}\left(\sum_k \alpha_k p_{kj}^n - \pi_j\right) &\sim f(n) \pi_j.
	\end{align*}
	Then applying this to $E(n)$, we obtain
	\begin{align*}
		E(n) 
		&\sim - \sum_{i} \sum_{j} p_{ij} \log p_{ij} \pi_j f(n) = - f(n) \sum_{j} \pi_j \sum_{i} p_{ij} \log p_{ij}.
	\end{align*}
	Define $C_j = -\sum_{i \in S} p_{ij} \log p_{ij} < \infty, \forall j$. Then there exists a $C = \sup C_j < \infty$, which implies
	\begin{align*}
		E(n) &\sim f(n)\sum_{j \in S} \pi_j C_j.
	\end{align*}
	Therefore,
	\begin{align*}
		\lim\limits_{n \rightarrow \infty} E(n) < \infty \iff \lim\limits_{n \rightarrow \infty} f(n) < \infty \iff \lim\limits_{n \rightarrow \infty} Q_{ij}^n < \infty.
	\end{align*}
\end{proof}

This result leads to the final theorem, which is the same as has been shown for Gaussian processes by Feutrill and Roughan~\cite{feutrill2021differential}.

\begin{thm1}{3}
	The mutual information between past and future of a Markov chain is infinite if and only if the Markov chain is LRD.
\end{thm1}

\begin{proof}[Proof of Theorem~\ref{thm:mutual_infinite}]
	Theorem~\ref{lem:excess_infinite} and Proposition 8 of Crutchfield and Feldman~\cite{crutchfield_feldman_2003} imply the result.
\end{proof}


\section{Conclusion}

We have shown that an alternate perspective of LRD, using information theoretic measures, applies in the case of LRD Markov chains. That is, that LRD Markov chains are those that share infinite information between the infinite future and infinite past. This supports the notion that the definition of LRD for Markov chains exhibits the ``right" behaviour. 

As is common with other definitions of LRD, it is characterised by slow convergence to quantities, such as sample mean, and we have shown in this section that this behaviour extends to an information-theoretic quantity of stochastic processes, the entropy rate. Most LRD processes that have been developed are defined on a continuous state space, so this is the first extension that the authors are aware of that analyse the information theoretic properties of LRD processes on discrete spaces. The behaviour is driven by the return time random variable, in particular its infinite second moment, and therefore this is the simplest discrete valued model where LRD behaviour occurs. 



\section*{Acknowledgement}
\noindent We wish to thank ACEMS, CSIRO/Data61 and Defence Science and Technology Group for supporting this work. Thanks to Giang Nguyen for her advice and editing of this manuscript.



%
%
%
%

\bibliographystyle{APT}
\bibliography{myBibliography}

\end{document}